\documentclass[11pt]{amsart} 

\usepackage[utf8]{inputenc} 

  \usepackage{fullpage}

\usepackage{graphicx} 

\usepackage{amsmath}
\usepackage{amssymb}
\usepackage{color}
\usepackage{amstext}
\usepackage{amsthm}
\usepackage{url}
\usepackage{booktabs} 
\usepackage{array} 
\usepackage{paralist} 
\usepackage{verbatim} 
\usepackage{subfig} 

\usepackage{fancyhdr} 

\title{The Kolakoski sequence and related questions about orbits}
\newcommand\blfootnote[1]{%
  \begingroup
  \renewcommand\thefootnote{}\footnote{#1}%
  \addtocounter{footnote}{-1}%
  \endgroup
}

\begin{document}
\maketitle
{\centering Bobby Shen\footnote{Bobby Shen,  Department of Mathematics, Massachusetts Institute of Technology,
Cambridge, Massachusetts, USA, \url{runbobby@mit.edu}} \par}

\newtheorem{theorem}{Theorem}[section]
\newtheorem{corollary}{Corollary}[theorem]
\newtheorem{lemma}[theorem]{Lemma}
\newtheorem{proposition}[theorem]{Proposition}
\newtheorem{conjecture}[theorem]{Conjecture}
 
\theoremstyle{definition}
\newtheorem{definition}[theorem]{Definition}
 
\theoremstyle{remark}
\newtheorem*{remark}{Remark}
\begin{abstract}
The Kolakoski sequence is the unique infinite sequence with values in $\{1, 2\}$ and first two term $1, 2, \ldots$ which equals the sequence of run-lengths of itself; we call this $K(1, 2).$ We define $K(m, n)$ similarly for $m+n$ odd. The focus of this paper is not on well-known conjectures about limiting densities but rather on conjectures which are more discrete in nature. 

We define two functions, $E_{m, n}$ and $C_{m, n},$ which are naturally encountered when studying iterated run-length expansion\footnote{These were probably introduced by V. Chvatal in \cite{Chvatal}, but we currently cannot find the paper.} We conjecture that a certain doubly infinite family of finite sequences $E_{1, n}\left(1^{2^j}, 1^{2j}\right)$ has odd length for all $j>0$ and even $n>0.$ We prove that this statement is equivalent to orbits of certain functions $C_{1, n}(1, -)$ being as large as possible. We empirically verify this for all even $n$ and $j \le 13.$ 
\end{abstract}

\blfootnote{2010 \textit{Mathematics Subject Classsifications}: 05A99\\ Keywords: Kolakoski sequence, Recurrence} 
\section{Introduction}
The Kolakoski sequence is the unique infinite sequence with values in $\{1, 2\}$ and first two terms $1, 2, \ldots$ which equals the sequence of run-lengths in the run-length encoding of itself. See \cite{RLE} for one definition of run-length encoding. The existence and uniqueness is relatively easy to prove. The Kolakoski sequence begins $1, 2, 2, 1, 1, 2, 1, 2, 2, 1, 2, 2, 1, 1, 2, \ldots.$ There are many open problems associated with the Kolakoski sequence. Perhaps the most famous conjecture is that the limiting density of ``$1$'' in the Kolakoski sequence equals one-half. 

The Kolakoski sequence can be defined for other ``alphabets'' $\{m, n\},$ where $m$ and $n$ are distinct positive integers. In particular, we define $K(m, n)$ to be the unique infinite sequence with values in $\{m, n\}$ whose first term is $m$ and which equals the sequence of run-lengths of itself. When $m+n$ is even, the sequence $K(m,n)$ is much easier to understand because $K(m,n)$ is realized as a fixed point of substitution rules. Therefore, we focus on the case in which $m+n$ is odd. 

In 1993, V. Chvatal proved that the superior and inferior limits of the density of $1$ in the first $n$ terms is bounded by $1/2 \pm 0.0084.$ Unfortunately, we currently cannot find the paper online. However, we believe that he proved the intermediate result that in an infinite sequence with values in $\{1, 2\}$ and whose first $20$ run-length encodings all have values in $\{1, 2\}$, the superior and inferior limits of the density of $1$ in the first $n$ terms is bounded by $1/2 \pm 0.0084.$\footnote{We believe this because the author did a math research project with Yongyi Chen and Michael Yan in Spring 2015 at MIT. We independently used this intermediate idea and found that the density is bounded by $1/2 \pm 0.00833,$ which we think is unlikely to coincide with Chvatal's bound by pure chance. The density is also bounded by $1/2 \pm 0.00334$ if one replaces ``20'' by ``26.'' These results require a fair amount of computation power and memory. See an semi-formal summary of our methods at \cite{CSY}.}

We think Chvatal proved this result by introducing functions $E_{m, n},$ $C_{m, n}$ at least for $(m,n) = (1,2).$ These functions are naturally encountered when discussing the Kolakoski sequence and iterated run-length encoding. We reproduce their definitions here. The functions $E_{m, n}, C_{m, n}$ have recursive definitions, and the functions may be interesting in their own right. 

One remarkable conjecture regarding the function $C_{m, n}$ is the following. First, we note that for all sequences $s,$ the function $C_{m, n}(s, -)$ is a length-preserving bijection over the set $\{m, n\}^*.$ Thus, it makes sense to discuss the orbit of this function. We conjecture that for all $j>0$ and even $n>0$, the orbit of the sequence $(1,1,1 \cdots 1,1,1)$ with $2j-1$ ones under the function $C_{1, n}(1, -)$ has size $2^{j}.$ On the other hand, we prove that the orbit of a sequence $t$ with length $2j-1$ under the function $C_{m, n}(s, -)$ divides $2^j.$ In this sense, the sequence $(1, 1, 1, \cdots, 1, 1, 1)$ has the maximum possible orbit length under $C_{1, n}(1, -).$ We also prove that this conjecture is related equivalent to the following: For all $j>0$ and even $n>0,$ the sequence $E_{1, n}\left(1^{2^j}, 1^{2j} \right)$ has odd length. In this sense, this conjecture states that a certain doubly infinite family of sequences all has odd length. This conjecture is fundamentally different than the usual conjectures about limiting densities in the Kolakoski sequence. 

Using a lot of computational power and a lemma, we have empirically verified this conjecture for all $j \le 13$ and even $n>0$. (The lemma reduces this particular infinite set of cases to a finite set.)




\subsection{Outline}
Section 2 introduces two auxiliary functions $C_{m, n}$ and $E_{m, n}$ which are naturally encountered when discussing the Kolakoski sequence and iterated run-length encoding. 
In section 3, we prove basic facts about these orbits, formulate our three main conjectures, and prove that two of them are equivalent.
Sectuib 4 discusses our algorithms for verifying many small cases  of these conjectures.

\subsection{Notation}
All numbers in this paper are positive integers except when otherwise specified. If $s$ is a sequence, then its terms are $s_1, s_2, \ldots.$ The length of $s$ is $|s|.$ Sometimes, we will use $s^{(1)}, s^{(2)},$ etc. to denote different sequences. If $s$ and $t$ are sequences with $s$ finite, then $st$ is the concatenation of $s$ (first) and $t.$ If $s$ is a finite sequence, $x$ is a number, and $p$ is a positive integer, then $s^p$ is the concatenation of $n$ copies of $s,$ and $x^p$ is the sequence of length $p$, all of whose terms are $x.$ 
The complement of a sequence with values in $\{m, n\}$ means the unique sequence of the same length with values in $\{m, n\}$ which is termwise different. The notation $\{m, n\}^k$ means the set of sequences of length $k$ with values in $\{m, n\}.$ The notation $\{m, n\}^k$ means the set of finite sequences with values in $\{m, n\}.$

 Henceforth, a \textit{subsequence} of a sequence $s$ will always mean a subsequence of consecutive elements. 

\section{Iterated run-length encoding and the functions $C_{m, n}$ and $E_{m, n}$}
Let $m, n$ be distinct positive integers. Recall that $K(m, n)$ is the unique infinite sequence with values in $\{m, n\}$ and first term $m$ which equals the sequence of run lengths of itself. It is clear that the infinite sequence $K(m, n)$ has local structure. Therefore, it is natural to try to express $K(m, n)$ as an infinite concatenation of finite sequences in a meaningful way. In this section, we develop one way of achieving this, and in doing so, we will define the functions $C$ and $E.$

We define the function $R,$ (the ``run-length'' function) such that for any possibly infinite sequence $s,$ $R(s)$ is the sequence of run lengths of $s.$ Note that $R(K(m, n)) = K(m, n).$

Because we are focusing on "alphabets" of size $2,$ the function $R$ almost has an inverse. To be precise, given $m \neq n$ and a posisbly infinite sequence of positive integers $t,$ there are exactly two different possible values of $s$ with values in $\{m, n\}$ such that $R(s) = t.$ The function $R \circ R$ also almost has an inverse. To be precise, given a possibly infinite sequence of positive integers $t,$ there are exactly four different possible values of $s$ such that both $s$ and $R(s)$ have values in $\{m, n\}$ and $R(R(s)) = t.$ These four values of $s$ are realized by first choosing a ``starting point'' for $R^{-1}(t)$, either ``m'' or ``n,'' then choosing a starting point for $R^{-1}(R^{-1}(t)),$ again either ``m'' or ``n'' and unrelated. This motivates the following definition.

\begin{definition}
Let $m \neq n.$ We define the function $E_{m, n}$, (the ``expansion'' function) which maps pairs of sequences to sequences as follows. Let $s$ be a possibly infinite sequence of positive integers, usually but not necessarily with values in $\{m, n\}.$ Let $t$ be a finite, nonempty sequence of positive integers with values in $\{m, n\}.$

If $t$ has length $1,$ then we define $E_{m, n}(s,t)$ to be the unique sequence with values in $\{m, n\}$ and whose first term equals $u_1$ and such that $R(E_{m, n}(s,t)) = t.$

If $t$ has length greater than $1,$ then let $t'$ be $t$ without its first term. Then we define $E_{m, n}(s,t)$ to equal $E_{m, n}(E_{m, n}(s,t_1), t'),$ where we regard the number $t_1$ as a length 1 sequence. In other words, $E_{m, n}(s,t)$ is the unique sequence such that 
\begin{itemize}
\item $R^{|t|}(E_{m, n}(s,t)) = s.$
\item For $k = 0, 1, \ldots, |t|-1.$ $R^k(E_{m, n}(s,t))$ has values in $\{m, n\}$.
\item For  $k = 0, 1, \ldots, |t|-1,$ the first term of $R^k(E_{m, n}(s,t))$ equals $t_{|t|-k}$. (Note the reversal.)
\end{itemize}
\end{definition}

\begin{remark} Informally, we call $E_{m, n}(s, t)$ the ``expansion of the sequence $s$ with starting points in $t$.'' Note that $t_1$ is the first starting point used when expanding, $t_2$ is the second, etc. We insist on having the subscripts in $E_{m, n}$ because later, we will discuss $E_{1, n}(1^{2^j}, 1^{2j}).$ $n$ can be an arbitrary integer greater than $1,$ and this expression depends on $n.$
\end{remark}

\begin{remark} Recall that $R(K(m, n)) = K(m, n).$ For any $k > 0,$ we have $R^k(K(m, n)) = K(m, n),$ and the first term of $K(m, n)$ equals $m.$ Therefore, $K(m, n) = E_{m, n}(K(m, n), m^k).$
\end{remark}

It is natural to consider what happens when we expand a concatenation of two sequences with a single set of starting points: $E_{m, n}(s^{(1)}s^{(2)}, t).$  First suppose that $t$ has length $1.$ It's easy to see that the function $u:= E_{m, n}(s^{(1)}s^{(2)}, t)$ is the concatenation of two sequences, $u = u^{(1)}u^{(2)},$ such that $R(u^{(i)}) = s^{(i)}$ for $i = 1, 2.$ The first term of $u^{(1)}$ equals the first term of $u,$ which is $t_1.$ The first term of $u^{(2)}$ is different from the last term of $u^{(1)}$; specifically, one term is $m,$ and the other is $n$ in an unspecified order. Note that the last term of $u^{(1)}$ only depends on $s^{(1)}$ and $t$; spefically, the last term of $u^{(1)}$ equals $t_1$ if $|s^{(1)}|$ is odd, and this last term equals $m+n-t_1$ if $|s^{(1)}|$ is even. Combining these observations, we have
\begin{align*}
E_{m, n}(s^{(1)}s^{(2)}, t) = E_{m, n}(s^{(1)}, t) E_{m, n}(s^{(2)}, m+n-u^{(1)}_{|u^{(1)}|}) \end{align*}

Next consider $E_{m, n}(s^{(1)}s^{(2)}, t)$ where $t$ has length $2.$ $E_{m, n}(s^{(1)}s^{(2)}, t_1)$ (where $t_1$ is regarded as a sequence of length $1$) can be characterized as in the previous paragraph. Let $u, u^{(1)}, u^{(2)}$ be as in the previous paragraph. It is easy to see that $v:=E_{m, n}(s^{(1)}s^{(2)},t_2)$ is the concatenation of two sequences, $v = v^{(1)}v^{(2)},$ such that $R(v^{(i)}) = u^{(i)}.$ The first term of $v^{(1)}$ equals the first term of $v,$ which equals $t_2.$ The first term of $v^{(2)}$ is different from the last term of $v^{(1)}.$ Note that $v^{(1)} = E_{m, n}(u^{(1)}, t_2)$ and $u^{(1)} = E_{m, n}(s^{(1)}, t_1).$ Therefore, $v^{(1)} = E_{m, n}(s^{(1)}, t).$ Also, the last term of $v^{(1)} =  E_{m, n}(s^{(1)}, t)$ only depends on $t_2, t_1,$ and $s^{(1)}$. (To be very precise, the last term of $v^{(1)}$ doesn't depend on $t_1,$ but we don't need to be that precise.) 

The discussion in the previous two paragraphs motivates the following definition and proposition.

\begin{definition}
We define the ``torsion'' function $C_{m,n}$ which maps pairs of finite sequences to finite sequences as follows. Let $m \neq n,$ $t$ be a finite sequence with values in $\{m, n\},$ and $s$ be a finite sequence, usually with values in $\{m, n\}$ as well. For $k = 1, 2, \ldots, |t|,$ let $t^{(k)}$ be the sequence (of length $k$) which is the first $k$ terms of $t,$ in order. We define $C_{m, n}(s, t)$ to be a sequence with values in $\{m, n\}$ of the same length as $t$ such that for $k = 1, 2, \ldots, |t|,$ the $k^{\text{th}}$ term of $C_{m, n}(s, t)$ equals $m+n$ minus the last term of $E_{m, n}(s, t^{(k)}).$ In other words, $C_{m, n}(s, t)$ is the complement of the sequence of the last terms in the intermediate expansions of $E_{m, n}(s, t).$

Informally, we call $C_{m, n}(s, t)$ the ``torsion of the sequence $t$ by the sequence $s$.''
\end{definition}

\begin{proposition} \label{EC}
Let $m \neq n,$ $t^{(1)}$ be a finite sequence with values in $\{m, n\},$ $s^{(1)}$ be a finite sequence of positive integers, and $s^{(2)}$ be an arbitrary sequence of positive integers. Then
\begin{align*} E_{m, n}(s^{(1)}s^{(2)}, t) = E_{m, n}(s^{(1)}, t) E_{m, n}(s^{(2)}, C_{m, n}(s^{(1)}, t)). \end{align*}
In other words, the expansion of the concatenation $s^{(1)}s^{(2)}$ by $t$ equals the expansion of $s^{(1)}$ by $t$ concatenated with the expansion of $s^{(2)}$ by the torsion of $t$ by $s^{(1)}.$

\end{proposition}
The proof of this proposition is basically induction on the length of $t.$

Here are some further useful propositions involving $C$ and $E,$ where we repeat the previous proposition for convenience.
\begin{proposition} Let $m \neq n,$ $s$ and $s^{(1)}$ be a finite sequence, $s^{(2)}$ be an arbitrary sequence,  and $t, t^{(1)}, t^{(2)}$ be finite sequences with values in $\{m, n\}.$
\begin{align}
	\label{z1} E_{m, n}(s^{(1)} s^{(2)}, t)&= E_{m, n}(s^{(1)}, t)E_{m, n}(s^{(2)}, C_{m, n}(s^{(1)}, t)). \\
	\label{z2} E_{m,n}(s, t^{(1)}t^{(2)}) &= E_{m, n}(E_{m, n}(s, t^{(1)}),t^{(2)}). \\
	\label{z3} C_{m, n}(s^{(1)} s^{(2)}, t) &= C_{m, n}(s^{(2)}, C_{m, n}(s^{(1)}, t)). \\
	\label{z4} C_{m, n}(s, t^{(1)} t^{(2)}) &= C_{m,n}(s, t^{(1)})C_{m, n}(E_{m, n}(s, t^{(1)}),t^{(2)}).
\end{align}
\end{proposition}

We omit the proofs of these statements. There are some semi-formal proofs provided in \cite{CSY}.

\section{A conjecture about orbits of $C_{m, n}.$}
For a fixed finite sequence $s$ and integer $k>0,$ the function $C_{m, n}(s, -)$ maps $\{m, n\}^k$ to itself. In this section, we discuss the nature of this function, such as bijectivity. We also formulate a conjecture about the orbit lengths of this function.

\begin{proposition} \label{jk} Let $m \neq n,$ $s$ be any finite sequence, and $k \ge j > 0.$ Let $t^{(1)}, t^{(2)} \in \{m, n\}^k.$ Then the sequence $t^{(1)}, t^{(2)}$ agree in their first $j$ terms iff the two sequences $C_{m, n}(s, t^{(1)}), C_{m, n}(s, t^{(2)})$ agree in their first $j$ terms.
\end{proposition}

\begin{proof}
The terms of the sequences $C_{m, n}(s, t^{(1)})$ are determined by the intermediate expansions of $E_{m, n}(s, t^{(1)}).$ The first $j$ digits of $C_{m, n}(s, t^{(1)})$ are determined by the first $j$ intermediate expansions of $E_{m, n}(s, t^{(1)}).$ Therefore, if $t^{(1)}, t^{(2)}$ agree in their first $j$ terms, then $C_{m, n}(s, t^{(1)}), C_{m, n}(s, t^{(2)})$ agree in their first $j$ terms.

On the other hand, suppose that $t^{(1)}, t^{(2)}$ agree in their first $j-1$ terms but not their first $j$ terms. As in the previous paragraph, $C_{m, n}(s, t^{(1)}), C_{m, n}(s, t^{(2)})$ agree in their first $j-1$ terms. The $j^{\text{th}}$ intermediate expansions of $E_{m, n}(s, t^{(1)})$ and $E_{m, n}(s, t^{(2)})$ are expansions of the same sequence, namely the common $(j-1)^{\text{th}}$ intermediate expansion of $E_{m, n}(s, t^{(i)})$ but with different starting points: One starting point is $m,$ and the other starting point is $n$ in an arbitrary order. Therefore, the two $j^{\text{th}}$ expansions are exactly ``complement sequences'' i.e. the two sequences have the same length and elementwise add to $m+n, m+n, m+n, \ldots.$ Therefore, these two sequences must have different last elements, and the two sequences  $C_{m, n}(s, t^{(1)}), C_{m, n}(s, t^{(2)})$ have different elements at index $j.$
\end{proof}

Setting $j=k,$ we see that $C_{m, n}(s, -)$ is a bijection on $\{m, n\}^k.$ Therefore, it makes sense to discuss the orbits and orbit lengths of $C_{m, n}(s, -).$ The following proposition shows that the orbit lengths must be powers of $2.$

\begin{proposition}  Let $m \neq n,$ $s$ be any finite sequence, and $k > 0.$ The orbit lengths of the map $C_{m, n}(s, -)$ on the set $\{m, n\}^k$ are all powers of $2.$ \end{proposition}

\begin{proof} For shorthand, let $f$ be the function $C_{m, n}(s, -).$ Suppose that $t^{(0)}, t^{(1)}, \ldots, t^{(p-1)}, t^{(p)} = t^{(0)}$ is one orbit under $f$ of length $p$. Write $p$ in the form $2^rq$ where $q$ is odd. Assume, for the sake of contradiction, that $p$ is not a power of $2.$ Then $q>1$ and $r<p.$

Consider the two sequences $t^{(0)}, t^{(r)}.$ Since $r<p,$ these two sequences are different elements of $\{m, n\}^k,$ so they do not agree in their first $k$ elements. Let $j$ be maximal so that these two sequences agree in their first $j$ terms. Then $j<k.$ By Proposition \ref{jk}, the two sequences $f(t^{(0)}), f(t^{(r)})$ agree in their first $j$ terms but not their first $j+1$ terms. These two sequences equal $t^{(1)}, t^{(r+1)}$ by construction. By induction, we have that for $i = 0, 1, \ldots, n-r,$ the two sequences $t^{(i)}, t^{(i+r)}$ agree in their first $j$ terms but not their first $j+1$ terms. In particular, consider $i = 0, r, 2r, \ldots, n-r.$ 

We have the sequences with values in $\{m, n\}^k$ $t^{(0)}, t^{(r)}, t^{(2r)}, \ldots, t^{(n-r)}, t^{(n)} = t^{(0)}.$ Any two consecutive elements agree in their first $j$ terms but not their first $j+1$ terms. Therefore, all terms have the same first $j$ terms, but their index $j+1$ terms alternate. The quantity $n/r$ is odd, so there is an odd number of alternations from $t^{(0)}$ to $t^{(n)}.$ This is a contradiction because these two sequences are actually equal (hence have the same index $j+1$ term). Therefore, $p$ must be a power of $2.$
\end{proof}

The following proposition uses another characterization of orbit lengths to give an upper bound.

\begin{proposition} \label{orbit}
let $m \neq n,$ $k > 0,$ $s$ be any sequence, and $t \in \{m, n\}^k.$ The size of the orbit of $t$ under the map $C_{m, n}(s, -)$ is at most $2^{\lceil k/2 \rceil}.$ 
\end{proposition}
\begin{proof}
We proceed by induction on $k.$ Our base cases are $k=1, 2.$ If $k=1,$ then the set $\{m, n\}^k$ has size $2,$ so an orbit has size at most $2,$ as desired. 

If $k=2,$ then we must show that $C_{m, n}(s, C_{m, n}(s, t)) = t.$ By identity (\ref{z3}), this is equivalent to showing that $C_{m, n}(ss, t) = t.$ Let $t' = C_{m, n}(ss, t).$ To show that the first terms of $t'$ and $t$ are equal, observe that $ss$ has an even number of terms, so $E_{m, n}(ss, t_1)$ has an even number of runs, so $t'_1 = t_1.$ For the second terms, observe that $E_{m, n}(ss, t_1) = E_{m,n}(s, t_1) E_{m, n}(s, C_{m, n}(s, t_1)).$ Both subsequences on the right hand side have the same number of terms (since they are either equal or complements), so the left hand side has an even number of terms, so $E_{m, n}(ss, t)$ has an even number of runs, so $t'_2 = t_2.$

Now assume that $k>2$ and that the proposition is true for smaller $k.$ Let $r = 2^{\lceil k/2 \rceil}.$ We must show that $C_{m, n}(s^r, t) = t.$ (This is a repeated application of identity (\ref{z3}).) Let $t'$ be the sequence $t$ excluding its last two terms so that $|t'| = k-2>0.$ By the inductive hypothesis, $C_{m, n}(s^{r/2}, t') = t'.$ Therefore, $C_{m, n}(s^{r/2} s^{r/2}, t') = t',$ and the sequence $C_{m, n}(s^r, t)$ agrees with $t$ in all terms except possibly the last two. To verify that the second-to-last terms are equal, we need to check that $E_{m, n}(s^r, t')$ has even length. Indeed, 
\begin{align*}E_{m, n}(s^r, t) &= E_{m, n}(s^{r/2} s^{r/2}, t) 
\\&= E_{m, n}(s^{r/2}, t) E_{m, n}(s^{r/2}, C_{m, n}(s^{r/2}, t)) 
\\&= E_{m, n}(s^{r/2}, t)^2.\end{align*}
To verify that the last terms are equal, we need to check that $E_{m, n}(s^r, t'')$ has even length, where $t''$ is the sequence $t$ excluding just its last term. Indeed, 
\[ E_{m, n}(s^r, t'') = E_{m,n}(E_{m, n}(s^r, t'), t_{|t|-1}) = E_{m, n}(E_{m, n}(s^{r/2}, t)^2, t_{|t|-1}),\]
and as in the $k=2$ case, an expansion of the form $E_{m, n}(u^2, x)$ has even length if $|x| = 1$.

\end{proof}

The next natural questions are if there are stronger bounds on the sizes of orbits in $\{m, n\}^k.$ Unfortunately, lwer bounds on orbits remain quite mysterious, but we have startling conjectures that imply that our upper bound is tight.

\begin{conjecture} \label{conj-orbit} Let $n$ be even, $j>0,$ and $t = 1^{2j-1}.$ Then the orbit of $t$ under the map $C_{1,n}(1, -)$ has length $2^j.$ \end{conjecture}

The following proposition provides a somewhat more concrete equivalent statement.
\begin{proposition} \label{conj-equivalence} Let $n$ be even. The conjecture \ref{conj-orbit} for the case $n$ is equivalent to the statement that the sequence $E_{1, n}\left(1^{2^j}, 1^{2j}\right)$ has odd length for all $j>0.$
\end{proposition}
\begin{proof} Fix $n$ even. First, we show that Conjecture \ref{conj-orbit} implies that  $E_{1, n}\left(1^{2^j}, 1^{2j}\right)$ has odd length for all $j>0.$ 



By Conjecture \ref{conj-orbit}, the orbit of $1^{2j+1}$ under $C_{1, n}(1, -)$ has length $2^{j+1},$ and the orbit of $1^{2j-1}$ has length $2^j.$ 

The orbit of $1^{2j}$ must be at least as big as the orbit of $1^{2j-1},$ which has length $2^j.$ By Proposition \ref{orbit}, the orbit of $1^{2j}$ has length at most $2^j.$ Therefore, the length is exactly $2^j.$ 

We have $C_{1, n}\left(1^{2^j}, 1^{2j+1} \right) \neq 1^{2j+1}$ because the orbit has length $2^{j+1},$ but by the previous paragraph, the left hand side is a sequence which begins with at least $2j$ ones, so $C_{1, n}\left(1^{2^j}, 1^{2j+1} \right) = 1^{2j} 2.$ The fact that the last term is flipped means that the penultimate expansion $E_{1, n}\left(1^{2^j}, 1^{2j} \right)$ has odd length, as desired.

Now, we show the converse. Assume that $E_{1, n}\left(1^{2^j}, 1^{2j}\right)$ has odd length for all $j>0.$ We now prove Conjecture \ref{conj-orbit} by induction on $j.$ The base case is $j=1,$ which is trivial.

Suppose that $j \ge 2$ and that Conjecture \ref{conj-orbit} is true for $j-1.$ We must show that the orbit of $1^{2j-1}$ has length $2^j.$ The orbit of $1^{2j-1}$ is at least as large as the orbit of $1^{2j-3},$ which has length $2^{j-1}$ by the inductive hypothesis. On the other hand, Proposition \ref{orbit} states that the orbit has length at most $2^j.$ Since the orbit length is a power of $2,$ it suffices to show that the orbit length does not divide $2^{j-1}.$ 

Assume for the sake of contradiction that this is not true. Equivalently, we are assuming that $C_{1, n} \left( 1^{2^{j-1}}, 1^{2j-1} \right) = 1^{2j-1}.$ Then
$u:= C_{1, n} \left( 1^{2^{j-1}}, 1^{2j} \right)$ begins with at least $2j-1$ ones. We compute
\[ C_{1, n} \left(1^{2^j}, 1^{2j} \right) = C_{1, n} \left( 1^{2^{j-1}}, 1^{2j} \right) C_{1, n} \left( 1^{2^{j-1}}, u \right).\]
Since $u$ is equal to either $1^{2j}$ or $1^{2j-1}n,$ the two subsequences on the right hand side are either identical or complements. In either case, the left hand side has even length. This contradicts our assumption that $C_{1, n} \left(1^{2^j}, 1^{2j} \right)$ has odd length. Therefore, in fact the orbit length of $1^{2j-1}$ is exactly $2^j,$ as desired.
\end{proof}

Here is a conjecture similar to Conjecture \ref{conj-orbit} which involves a generalization to $m=-1.$ To formulate the $m=-1$ case, we must first explain in what sense negative values of $m, n,$ and/or terms in the sequence $s$ are well-defined. 

Let $m \neq n$ be arbitrary integers. We will only be considering the function $C_{m, n}$, not $E_{m, n}.$ We wish to define $C_{m, n}(s, t)$ where $t$ is a finite sequence with values in $\{m, n\}.$ We can first define $C_{m, n}(x, t)$ for possibly negative integers $x$ (or sequences of length $1$) then use identity (\ref{z3}). 

If $t$ has length $0,$ then we define $C_{m, n}(x, t):= t.$ Otherwise, suppose that $t$ is empty, and let $t'$ be the sequence $t$ without its first term, so that $t = t_1 t'.$ By identity (\ref{z4}), 
\[ C_{m, n}(x, t_1 t') = C_{m,n}(x, t_1)C_{m, n}(E_{m, n}(x, t_1),t').\]

The first subsequence, $C_{m,n}(x, t_1)$, is easily seen to be the single term $m+n-t_1,$ or the complement of $t_1$. In the $x>0$ case, it makes sense to write $E_{m, n}(x, t_1) = t_1^x.$ The by identity (\ref{z3}), we have 
\[C_{m, n}(E_{m, n}(x, t_1),t') = C_{m, n}\left( t_1^x, t'\right) = C_{m, n}\left(t_1, -\right)^x(t').\]
Observe that the right hand side makes sense for arbitrary $m, n, x \in \mathbb{Z}$ provided that we define $C_{m, n}(x, u)$ inductively based on the length of $u.$ Indeed, this is our definition.
\begin{definition}
Let $m, n$ be distinct integers, possibly nonpositive, $s$ be a finite sequence of integers, possibly nonpositive, and $t$ be a finite sequence with values in $\{m, n\}.$ We define $C_{m, n}(s, t)$ as follows.

We first define $C_{m, n}(x, t)$ for integers $x$ then use identity (\ref{z3}) to define $C_{m, n}(s, t).$ We define $C_{m, n}(x, t)$ inductively based on the length of $t.$ If $t$ has length $0,$ then the output is also the empty sequence. Otherwise, write $t = t_1 t'.$ Then we define
\[ C_{m, n}(x, t_1 t') := (m+n-t_1)C_{m, n}(t_1, -)^x(t').\]
Note that if $x$ is negative, we are using the fact that $C_{m, n}(t_1, -)$ is a length-preserving bijection.
\end{definition}

With this definition, we now formulate the following conjecture.

\begin{conjecture}\label{conj-orbit2} Let $n$ be even and possibly nonpositive. Let $j > 0.$ Let $v$ be the sequence $n, -1.$ Let $t$ be either the sequence $-1 v^{j-1}$ or the sequence $v^{j-1} n,$ which are both length $2j-1$ sequences. Then the orbit of $t$ under the map $C_{-1, n}(-1, -)$ has length $2^j.$
\end{conjecture}

One can show that the upper bound in Proposition \ref{orbit} applies to the nonpositive integer case. This conjecture is not associated with an odd-expansion analog like Proposition \ref{conj-equivalence}.

\section{Empirical evidence for the conjectures for $j \le 13.$}
The two conjectures stated above are quite mysterious to us, and in fact our only compelling reason for believing them is our extensive empirical evidence. A much less compelling reason is that both conjectures are true for $n=0.$ In the $n=0$ case, the map $C_{m, n}(0, -)$ is the identity map, so the computations become much simpler, and the claims reduce to straightforward induction arguments. We omit the details.

First, we show that certain cases $(n,j)$ of Conjectures \ref{conj-orbit} and \ref{conj-orbit2} equivalent.

\begin{proposition} \label{mod-equivalence} Let $j>0$ and $n_0$ be even. Then for all even integers $n$ such that $2^{j-1}$ divides $n-n_0,$ the cases $(n,j)$ of Conjecture \ref{conj-orbit} are equivalent to each other. Also, for all such $n,$ the cases $(n,j)$ of Conjecture \ref{conj-orbit2} are equivalent to each other. We make no claim of equivalence between the two conjectures nor any claim that the conjectures are true in these cases.
\end{proposition}

\begin{proof}Fix $m = \pm 1.$ Let $S$ be the set of even integers $n$ such that $n \equiv n_0 \pmod{2^{j-1}}.$ In the rest of this proof, $n$ is restricted to be in $S.$

Let $F$ be a function such that for any sequence $s,$ $F(s)$ replaces all occurrences of terms which are not equal to $m$ with $0.$ Let $G_n$ be a function such that for any sequence $s$ with values in $\{m, 0\}$ and replaces all occurrences of $0$ with $n.$

We claim that if $1 \le k \le 2j-1$, $x = m$ or $n,$ and $t$ is in $\{m, 0\}^{k}$, then the value $F(C_{m, n}(x, G_n(t)))$ is independent of $n$ We proceed by induction on $k.$ The base case, $k = 1,$ is easy since $C_{m, n}(x, -)$ switches $m$ and $n$ for all $n.$

Suppose that $F(C_{m, n}(x, G_n(t)))$ is independent of $n$ over the set $\{m, 0\}^{k-1}.$ Let $t$ be in $\{m, 0\}^{k-1}$ and $n$ be in $S.$ The first term of $F(C_{m, n}(x, G_n(t)))$ is independent of $n$ because $C_{m, n}(m, u)$ always "flips" the first term of $u.$ It remains to check that terms after the first are independent of $n.$

Let $H(s)$ be $s$ without its first term. By identity (\ref{z4}), the sequence $C_{m, n}(x, G_n(t))$, excluding the first term, equals
\[ C_{m,n}(G_n(t)_1, -)^x (H(G_n(t))) = C_{m, n}(G_n(t_1), -)^x (G_n(H(t))).\]
The sequence $F(C_{m, n}(x, G_n(t)))$, excluding the first term, equals
\[F\left( C_{m, n}(G_n(t_1), -)^x (G_n(H(t))) \right) = \left( F(C_{m, n}(G_n(t_1), G_n(-)) \right)^x  ( H(t)).\]
Note that either $G_n(t_1) = m$ for all $n$ or $G_n(t_1) = n$ for all $n.$ Also, either $x=m$ for all $n$ or $x=n$ for all $n.$

If $x=m$ for all $n,$ then the function $\left( F(C_{m, n}(x, G_n(-))\right)^x$ becomes $\left( F(C_{m, n}(G_n(t_1), G_n(-)) \right)^m,$ which is independent of $n$ when applied to $H(t)$ because $H(t)$ has length $k-1$ and the "base" $\left( F(C_{m, n}(G_n(t_1), G_n(-)) \right)$ is independent of $n$ over length $k-1$ sequences. 

If $x = n$ for all $n,$ we must prove that $\left( F(C_{m, n}(G_n(t_1), G_n(-)) \right)^n,$ is independent of $n$ when applied to length $k-1$ sequences, but we have more work to do because the number of iterations depends on $n.$ We now use Proposition \ref{orbit}, which states that an orbit $C_{m, n}(s, -)$ has length dividing $2^{\lceil |s|/2 \rceil}.$ In this case, we are concerned with length $k-1$ sequences, and $k \le 2j-1,$ so orbits have length dividing $2^{\lceil (2j-2)/2 \rceil} = 2^{j-1}.$ Therefore, $\left( F(C_{m, n}(G_n(t_1), G_n(-)) \right)^{2^{j-1}}$ is the identity. Since values of $n$ in $S$ all have the same residue $\bmod 2^{j-1},$ $\left( F(C_{m, n}(G_n(t_1), G_n(-)) \right)^n$ is independent of $n$ over length $k-1$ sequences. This completes the proof that $F(C_{m, n}(x, G_n(t)))$ is independent of $n$ over length $k$ sequences.

In the above claim, we now specialize to $k = 2j-1, x= m$. Note that for all $n,$ $F$ and $G_n$ are inverses betwen $\{m,n\}^{2j-1}$ and $\{m, 0\}^{2j-1}.$ We can intertwine these bijections with the set automorphism $C_{m, n}(m, -)$ over $\{m, n\}^{2j-1}$ to get some other automorphism $F \circ C_{m, n}(m, -) \circ G_n$ on the set $\{m, 0\}^{2j-1}.$ These two set automorphisms are essentially equivalent. In particular, the length of the orbit of $s$ under the first automorphism equals the length of the orbit of $F(s)$ under the second automorphism. Since $F(C_{m, n}(m, G_n(t)))$ is independent of $n$, the orbit lengths of $C_{m, n}$ are independent of $n$ over length $2j-1$ sequences, so the statement of Conjecture \ref{conj-orbit} (if $m=1$) or Conjecture \ref{conj-orbit2} (if $m= -1$) is independent of $n.$ 
\end{proof}

%
%
%
%
%
%
%
%
%
%
%
%
%
%
\subsection{Explicit computation to verify the conjectures}

Everything in this section is implemented in C++.

Fix an integer $n$ in $\{2, 4, 6, 8, \ldots, 4094\}.$ In this subsection, we discuss how we explicitly compute $C_{1, n}(1, t)$ and $C_{-1, n}(-1, t)$ for $|t| \le 25$ in order to verify Conjectures \ref{conj-orbit} and \ref{conj-orbit2}.

We recursively compute and completely store the functions $C_{1, n}(1, -)$ and $C_{1, n}(n, -)$ for arguments of length at most $25$. This data is stored as two vectors of vectors of unsigned integers, $map\_m$ and $map\_n$ with the following convention. 

\begin{definition} 
We define the function $F$ as follows. The domain of $F$ is pairs $(j,k)$ with $k>0$ and $0 \le j < 2^k.$ For $k>0$ and $j \in \{0, \ldots, 2^k-1\}$, $F(j, k)$ is given by converting $j$ to a binary string $j_1$ padded to length $k,$ reversing $j_1$ to form the binary string $j_2,$ turning $j_2$ into the sequence $j_3,$ and replacing the values $\{0, 1\}$ with the values $\{1, n\}$ respectively in $j_3$ to form the sequence $j_4.$

Thus $F(-, k)$ is a bijection from $\{0, \ldots, 2^k-1\}$ to $\{1, n\}^k.$ We define $G(t)$ to be a sort of inverse: If $t$ is in $\{1, n\}^*,$ then $G(t) = j$ where $F(j,|s|) = s.$
\end{definition}

\textbf{Example.} Let $k = 5,$ $j = 13.$ The integer $j$ as a binary string of length $k$ is $01101,$ so $j_1 = 01101.$ Then $j_2 = 10110,$ and $j_3$ is the sequence $1,0,1,1,0.$ Finally, $j_4$ is the sequence $n,1,n,n,1.$ Likewise, $G(n,1,n,n,1) = 13.$

\begin{remark} Note that there is a reversal of bits. Of course, the algorithm would work fine if one consistently did not reverse bits. \end{remark}

For $k \le 25,$ and $j \in \{0, \ldots, 2^{k}-1\},$ $map\_m[k][j]$ equals $G(C_{1, n}(1, F(j, k))),$ and $map\_n[k][j]$ equals $G(C_{1, n}(n, F(j, k))).$ Note that both of these values are integers.

The recursion formulas are essentially identity (\ref{z4}). In particular, let $t$ be a sequence of length less at most $25$, and let $t'$ be $t$ without its first term. We have 
\begin{align*}C_{1, n}(1, t_1 t')&= (1+n-t_1)C_{m, n}(E_{m, n}(1, t_1), t')&= (1+n-t_1)C_{m, n}(t_1, t').\\
C_{1, n}(n, t_1 t') &= (1+n-t_1)C_{m, n}(E_{m, n}(n, t_1, t') &= (1+n-t_1)C_{m, n}(t_1, -)^n(t'). \end{align*}

The relations among sequences are readily converted to relations among integers. For example, $G(t')$ is an integer which is given by the integer quotient $G(t)/2.$ The hardest part is dealing with $C_{m, n}(t_1, -)^n.$ Evaluating this the naive way introduces a factor of $n$ into the runtime, which is too much. Instead, we decompose the permutation induced by $C_{m, n}(t_1, -)$ on $\{1, 2\}^k$ into cycles in time $O(2^k).$ We can then exponentiate a cyclic permutation in time which grows negligibly with $n.$ 

We then compute the length of the orbit of $1^k$ for $k \le 25$ under $C_{1, n}(1, -)$ by repeatedly calling $map\_m[k][-].$ As such, we have empirically verified Conjecture \ref{conj-orbit} for all even $n$ with $0 \le n < 4096$ and all $j \le 13.$ In view of Proposition \ref{mod-equivalence}, we have also verified the conjecture for all even $n$ and all $j \le 13.$  We have also performed the same empirical verifications for Conjecture \ref{conj-orbit2}. Note that when computing for Conjecture \ref{conj-orbit2}, we must compute inverses of maps. This is readily done using the cycle decomposition.

We have also used ad-hoc methods to verify the conjecture for $n=2$ and $j \le 23.$ Essentially, we expand $E_{1, 2}\left(1^{2^h}, 1^{2h}\right)$ for some small $h$ as an intermediate step. This is not feasible for $n \gg 2$ because the intermediate sequences $E_{1, n}\left(1^{2^h}, 1^{2h} \right)$ would be too long to be useful.

We conclude by admitting that verifying the conjectures for larger $j,$ even $j = 15,$ would use all of our random access memory.

\section{Acknowledgements}

The author originally researched this problem with Yongyi Chen and Michael Yan as a part of MIT's class \textit{Math Project Lab} in Spring 2015. Our main paper for this project is at \cite{CSY}. Conjecture \ref{conj-orbit} was first observed by Yongyi Chen for $n=2.$


\begin{thebibliography}{9}


\bibitem{CSY} Chen, Yongyi, Bobby Shen, and Michael Yan. \url{https://www.overleaf.com/4401114cgstjy} 

\bibitem{Chvatal} Chvatal, Vasek. ``Notes on the Kolakoski Sequence.'' Technical Report 93-84. DIMACS. 
\url{http://dimacs.rutgers.edu/TechnicalReports/abstracts/1993/93-84.html}. Unfortunately, we currently cannot find the actual paper.

\bibitem{RLE}
\url{http://mathworld.wolfram.com/Run-LengthEncoding.html}


 

\end{thebibliography}
\end{document}